\documentclass{amsart}

\usepackage[T1]{fontenc}
\usepackage{lmodern}
\usepackage{amsmath,amssymb,amsthm,mathtools}
\usepackage{microtype}
\usepackage{tikz}
\usetikzlibrary{arrows.meta,calc}
\usepackage[colorlinks=true,linkcolor=blue,citecolor=blue,urlcolor=blue]{hyperref}

\newtheorem{theorem}{Theorem}
\newtheorem{corollary}{Corollary}
\newtheorem{proposition}{Proposition}
\newtheorem{lemma}{Lemma}

\newcommand{\N}{\mathbb N}
\newcommand{\R}{\mathbb R}
\newcommand{\eps}{\varepsilon}
\newcommand{\claimqed}{\hfill$\triangleleft$}

\title[Close relatives of Hilbertian balls]{Homeomorphism between close relatives of Hilbertian balls}
\author{Antonio Avilés}
\date{}

\subjclass[2020]{54B10, 46B26, 54D30}
\keywords{Hilbertian ball, uniform Eberlein compact, norm-preserving homeomorphism}

\thanks{Author supported by Fundaci\'{o}n S\'{e}neca - ACyT Regi\'{o}n de Murcia project 21955/PI/22, and by ERDF and MICIU/AEI/10.13039/501100011033 project PID2021-122126NB-C32.}

\begin{document}

\begin{abstract} We present a solution to some problems posed by the author and Kalenda. We show that the closed ball of nonseparable Hilbert space in its weak topology is homeomorphic to its positive part, as well as to its product with the Hilbert cube. In the separable setting we obtain that there is a weak homeomorphism of the closed unit ball of $\ell_2$ onto its positive part that preserves the norm, and via a result of Dijkstra and van Mill, the same is true for the ball of $\ell_\infty=\ell_1^*$ in the weak$^*$ topology. All spaces $B(\kappa,a,b)$ considered by the author and Kalenda are shown to be homeomorphic. The solution has been found by AI (Chatgpt 5.5), the role of the author has been to ask the right questions, check and understand the answers, and adapt the writing to his personal human taste. 
\end{abstract}

\maketitle

\section{Introduction}

For a set \(\kappa\), that can be taken a cardinal, we consider the compact space
\[
B(\kappa)=
\left\{x\in[-1,1]^\kappa:\sum_{i\in\kappa}|x_i|\le1\right\},
\]
with the inherited product topology. It can also be viewed as the closed unit ball of the dual Banach space $c_0(\kappa)^*=\ell_1(\kappa)$ in the weak$^*$ topology. Through the obvious bijection taking the appropriate power of each coordinate and keeping its sign, $B(\kappa)$ is homeomorphic to the closed unit ball of the Banach space $\ell_p(\kappa)$ in the weak topology, for $1<p<\infty$. In particular, it is the closed unit ball of a Hilbert space in its weak topology. The class of compact spaces homeomorphic to subspaces of some $B(\kappa)$ is a well studied class known as uniform Eberlein compact spaces. In the countable case, by virtue of Keller-Klee's Theorem~\cite{Keller,Klee}, the space $B(\mathbb{N})$ is homeomorphic to any metrizable infinite-dimensional compact convex set in a locally convex space. However, in the uncountable case, even minor alterations of the space $B(\kappa)$ that may seem to share the most conspicuous properties can nevertheless be not homeomorphic. In the article \cite{Aviles2007}, using a result of Bell \cite{Bell}, it was shown that $B(\kappa)$ is a polyadic space, while there were renormings of the nonseparable Hilbert spaces whose closed balls are not polyadic in the weak topology. In a joint work with O. Kalenda \cite{AvilesKalenda2009}, we showed, using so-called \emph{fiber orders}, that the finite powers $B(\kappa)^n$ are all nonhomeomorphic for different $n$. In fact, as shown later using variants of chain conditions~\cite{Aviles2009}, $B(\kappa)^n$ is not a continuous image of $B(\kappa)^m$ if $n>m$. However, some relatives of the Hilbert ball remained elusive to all these techniques. It remained open whether the positive part of the ball $B^+(\kappa) = B(\kappa)\cap [0,1]^\kappa$ is homeomorphic to $B(\kappa)$. The space $B^+(\kappa)$ can be also viewed as the space of probability measures on the one point compactification of the discrete space $\kappa$. While a result of Kalenda \cite{Kalenda2008} implies that the product $B^+(\kappa)\times [0,1]$ is homeomorphic to $B^+(\kappa)$, it also remained open whether $B(\kappa)\times [0,1]$ is homeomorphic to $B(\kappa)$. These problems were reviewed in \cite{AvilesKalenda2010}, where a more general family of relatives of $B(\kappa)$ is introduced:  For functions
\[
 a,b:\kappa\longrightarrow[-1,1],\qquad a(i)<b(i),
\]
define the compact space
\[
B(\kappa,a,b)=
\left\{x\in\R^\kappa:
\sum_{i\in\kappa}|x_i|\le1,
\ a(i)\le x_i\le b(i)\text{ for all }i\in\kappa
\right\},
\]
obtained by cutting \(B(\kappa)\) with coordinate bands. We make the nontriviality assumption that $\mu(a,b)=\sum_{i\in\kappa}\bigl(a(i)^+ + b(i)^-\bigr)<1$, which is exactly when $B(\kappa,a,b)$ is an infinite compact space. The positive part of the ball $B^+(\kappa)$ is the particular case of the constant functions $a=0$ and $b=1$. The question was posed in \cite{AvilesKalenda2010} whether, for a given $\kappa$, all the nontrivial spaces $B(\kappa,a,b)$ are homeomorphic or not. We solve this here in the positive.

\begin{theorem}\label{thm:classification}
$B(\kappa,a,b)$ is homeomorphic to $B(\kappa)$ for any nontrivial choice of $
a,b:\kappa\to[-1,1]$ with $a(i)<b(i)$ for $i\in\kappa$ and $\mu(a,b)<1$.
\end{theorem}

This answers in particular the cases singled out in Problem~3 of
\cite[p.~339]{AvilesKalenda2010}.  If \(t^*\) denotes the function on \(\kappa\)
which is equal to \(t\) except at one point, where it is \(0\), we get the half ball $B^{1/2}(\kappa)=B(\kappa,-1^*,1)$ made of the elements of the ball with a fixed nonnegative coordinate.

\begin{corollary}\label{cor:problem3}
Let \(\kappa\) be an uncountable cardinal.  For every \(0<t<1\), the
compact spaces
\[
B(\kappa),\quad
B^+(\kappa),\quad
B^{1/2}(\kappa),\quad
B(\kappa,-t,t),\quad
B(\kappa,0,t)
\]
are mutually homeomorphic.
\end{corollary}

By the aforementioned result of Kalenda, the product of these spaces with the unit interval will remain homeomorphic to them. In fact, a little more is true.

\begin{corollary}\label{cor:interval-product}
	For every infinite cardinal \(\kappa\), $B(\kappa)$ is homeomorphic to $B(\kappa)\times[0,1]^\N$.
\end{corollary}

\begin{proof}
	Choose numbers \(r_n>0\) with
	\[
	\rho=\sum_{n=1}^\infty r_n<1.
	\]
	Let \(Q_r=\prod_{n=1}^\infty[0,r_n]\).  Since \(\sum r_n<\infty\), the
	function
	\[
	s(u)=\sum_{n=1}^\infty u_n
	\qquad(u\in Q_r)
	\]
	is continuous for the product topology.  Define
	\[
	\Theta:B(\kappa)\times Q_r\longrightarrow
	\left\{(z,u):z\in[-1,1]^\kappa,
	u\in Q_r,
	\sum_{i\in\kappa}|z_i|+\sum_{n=1}^\infty u_n\le1
	\right\}
	\]
	by
	\[
	\Theta(x,u)=((1-s(u))x,u).
	\]
	Since \(s(u)\le\rho<1\), this is a homeomorphism.  Its inverse is
	\[
	(z,u)\longmapsto\left(\frac{z}{1-s(u)},u\right).
	\]
	The target is a space of the form \(B(\Gamma,a,b)\), where
	$\Gamma$ is the disjoint union of $\kappa$ and $\N$, the coordinates in \(\kappa\) have interval
	\([-1,1]\), and the \(n\)-th new coordinate has interval \([0,r_n]\).
	Because \(|\Gamma|=\kappa\),
	Theorem~\ref{thm:classification} says that the target is homeomorphic to
	\(B(\kappa)\). We conclude that $B(\kappa)$ is homeomorphic to $B(\kappa)\times Q_r$.
\end{proof}

The first step of the proof of Theorem~\ref{thm:classification}, carried out in Section~\ref{sectionplane}, is to consider a certain piecewise linear map in two dimensions that will be later used as a basic brick of the construction. The second step, done in Section~\ref{sectkey}, is to produce a homeomorphism in the countable case with the extra property that norm is preserved. This result, which requires an extra hypothesis of divergent bounds, has an independent interest:

\begin{theorem}\label{thm:separable}
	Given $a,b:\mathbb{N}\longrightarrow [-1,1]$ with $a\leq 0\leq b$ and \[\sum_n \max\{-a(n),b(n)\}=\infty,\] there exists a homeomorphism $F:B^+(\mathbb{N})\longrightarrow B(\mathbb{N},a,b)$ such that $$\sum_n |F(x)_n|= \sum_n x_n$$ for all $x\in B^+(\mathbb{N})$.
\end{theorem}

\begin{corollary}
	There exists a homeomorphism $F:B^+(\mathbb{N})\longrightarrow B(\mathbb{N})$ such that $\sum_n |F(x)_n|= \sum_n x_n$ for all $x\in B^+(\mathbb{N})$.
\end{corollary}
We recall the result of Dijkstra and van Mill \cite{DijkstraVanMill}, cf. \cite[Theorem 4]{AvilesKalenda2010}, that there exists a homeomorphism $G:B(\mathbb{N})\longrightarrow [-1,1]^\mathbb{N}$ for which $\sup_n |G(x)_n| = \sum_n |x_n|$ and $G(B^+(\mathbb{N}))= [0,1]^{\mathbb{N}}$.

\begin{corollary}
There exists a homeomorphism $G:[0,1]^\mathbb{N}\longrightarrow [-1,1]^\mathbb{N}$ such that $\sup_n|G(x)_n| = \sup_n x_n$ for all $x\in [0,1]^\mathbb{N}$.  
\end{corollary}
 
Finally, in Section~\ref{sectglue}  we use Theorem~\ref{thm:separable} to glue countable blocks of coordinates together to pass to the uncountable case, finishing the proof of Theorem~\ref{thm:classification}.

%
%

\section{The planar map}\label{sectionplane}

Fix numbers $
0\le\alpha\le\beta\le1$ with $\beta>0$, and
\[
T=\{(c,x)\in[0,1]^2:c+x\le1\},
\]
\[
T_{\alpha,\beta}=\{(y,d)\in[-\alpha,\beta]\times[0,1]: |y|+d\le1\}.
\]
We are going to define a bijection $\Lambda_{\alpha,\beta}:T\longrightarrow T_{\alpha,\beta}$ that preserves the $\ell_1$-norm. For a fixed $r\in [0,1]$, $\Lambda_{\alpha,\beta}(c,x)=(y,d)$ is going to be the natural bijection between the segment $c+x=r$ inside $T$ and the union of two segments $|y|+d = r$ inside $T_{\alpha,\beta}$.

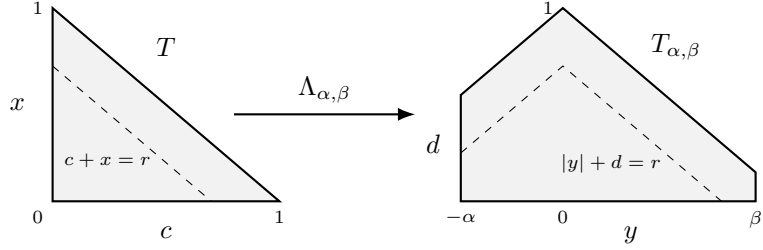
	\begin{center}
	\begin{tikzpicture}[x=3.0cm,y=2.55cm,>=Latex]
	\pgfmathsetmacro{\al}{0.45}
	\pgfmathsetmacro{\be}{0.85}
	\pgfmathsetmacro{\rr}{0.70}
	
	\begin{scope}
	\fill[gray!10] (0,0)--(1,0)--(0,1)--cycle;
	\draw[thick] (0,0)--(1,0)--(0,1)--cycle;
	\draw[thin,dashed] (0,\rr)--(\rr,0);
	\node at (0.5,0.8) {$T$};
	\node[below] at (0.5,-0.08) {$c$};
	\node[left] at (-0.08,0.5) {$x$};
	\node[below left] at (0,0) {\scriptsize $0$};
	\node[below] at (1,0) {\scriptsize $1$};
	\node[left] at (0,1) {\scriptsize $1$};
	\node[above right] at (0.01,0.13) {\scriptsize $c+x=r$};
	\end{scope}
	
	\draw[->,thick] (0.8,0.45)--(1.6,0.45);
	\node at (1.2,0.58) {$\Lambda_{\alpha,\beta}$};
	
	\begin{scope}[shift={(2.25,0)}]
	\fill[gray!10] (-\al,0)--(-\al,{1-\al})--(0,1)--(\be,{1-\be})--(\be,0)--cycle;
	\draw[thick] (-\al,0)--(-\al,{1-\al})--(0,1)--(\be,{1-\be})--(\be,0)--cycle;
	\draw[thin,dashed] (-\al,{\rr-\al})--(0,\rr)--(\rr,0);
	\node at (0.5,0.8) {$T_{\alpha,\beta}$};
	\node[below] at (0.3,-0.08) {$y$};
	\node[left] at (-\al-0.05,0.3) {$d$};
	\node[below] at (-\al,0) {\scriptsize $-\alpha$};
	\node[below] at (0,0) {\scriptsize $0$};
	\node[below] at (\be,0) {\scriptsize $\beta$};
	\node[left] at (0,1) {\scriptsize $1$};
	\node[above right] at (-0.05,\rr-0.6) {\scriptsize $|y|+d=r$};
	\end{scope}
	\end{tikzpicture}
\end{center}

For each $r\in [0,1]$, in the segment of $c+x=r$ inside $T$ the variables $x$ and $c$ range in the interval $[0,r]$. But the range $y$ for which $|y|+d=r$ is inside $T_{\alpha,\beta}$ is $[-r,r]$ if $0\leq r\leq \alpha$, $[-\alpha,r]$ if $\alpha\leq r\leq \beta$ and $[-\alpha,\beta]$ if $\beta\leq r\leq 1$. Thus, the formula for $\Lambda_{\alpha,\beta}(c,x) = (y,d)$ is
$$y(c,x) = \begin{cases} -r + 2c & \text{ if } 0\leq r <\alpha\\  -\alpha + \frac{\alpha+r}{r}c & \text{ if } \alpha\leq r \leq \beta \\-\alpha + \frac{\alpha+\beta}{r}c & \text{ if } \beta < r \leq 1\end{cases}, \ \ d(c,x) = r -|y(c,x)|, \ \ \text{ where } r=x+c,$$
with the exceptional case $\Lambda_{\alpha,\beta}(0,0)=(0,0)$ when $r=0$. 
A more synthetic formula for the first coordinate is $y(c,x)=-\ell(r)+\frac{\ell(r)+p(r)}{r}\,c$ where $\ell(r)=\min\{\alpha,r\}$, $p(r)=\min\{\beta,r\}$.

The inverse of $\Lambda_{\alpha,\beta}$ will be denoted by
\[
\Omega_{\alpha,\beta}=(A_{\alpha,\beta},B_{\alpha,\beta}):
T_{\alpha,\beta}\longrightarrow T.
\]
If \((y,d)\in T_{\alpha,\beta}\) and \(r=|y|+d\), we have
\(\Omega_{\alpha,\beta}(0,0)=(0,0)\), and for \(r>0\), the formulas derived from isolating the variables are:
\[
A_{\alpha,\beta}(y,d)=
r\,\frac{y+\ell(r)}{\ell(r)+p(r)},
\qquad
B_{\alpha,\beta}(y,d)=r-A_{\alpha,\beta}(y,d).
\]

\begin{proposition}\label{prop:planar-homeomorphism}
The map \(\Lambda_{\alpha,\beta}\) is a homeomorphism from \(T\) onto
\(T_{\alpha,\beta}\), with inverse \(\Omega_{\alpha,\beta}\).  Moreover it
preserves the $\ell_1$-norm:
\[
|y|+d=c+x
\qquad\text{whenever }(y,d)=\Lambda_{\alpha,\beta}(c,x).
\]
\end{proposition}

\begin{proof}
The preservation of the norm and the fact that they are inverses of each other is obvious by construction. The continuity at every point different from $(0,0)$ is clear from the synthetic formulas of the coordinates which are valid where $r\neq 0$. The continuity at $(0,0)$ follows from the preservation of the norm. 
\end{proof}

We will need the following estimate:

\begin{proposition}\label{prop:lipschitz-planar}
For fixed \(y\in[-\alpha,\beta]\), the function
\[
d\longmapsto A_{\alpha,\beta}(y,d),
\qquad 0\le d\le1-|y|,
\]
is \(\lambda_{\alpha,\beta}(y)\)-Lipschitz, where
\[
\lambda_{\alpha,\beta}(y)=
\max\left\{\frac12,\frac{y+\alpha}{\alpha+\beta}\right\}.
\]
In particular, \(0\le\lambda_{\alpha,\beta}(y)\le1\), and if
\(y\le\beta/2\), then
\[
\lambda_{\alpha,\beta}(y)\le\frac34.
\]
\end{proposition}

\begin{proof}
Put \(r=|y|+d\).  Since \(r\) has slope \(1\) as a function of \(d\), it
is enough to look at \(A_{\alpha,\beta}\) as a function of \(r\).

If \(0\le r\le\alpha\), then \(\ell(r)=p(r)=r\), and
\[
A_{\alpha,\beta}(y,d)=\frac{r+y}{2},
\]
so the slope is \(1/2\).  If \(\alpha\le r\le\beta\), then
\(\ell(r)=\alpha\), \(p(r)=r\), and
\[
A_{\alpha,\beta}(y,d)=r\frac{y+\alpha}{r+\alpha}.
\]
The slope is
\[
\frac{\alpha(y+\alpha)}{(r+\alpha)^2}\le \frac{\alpha}{r+\alpha}\le\frac12.
\]
 Finally,
if \(r\ge\beta\), then \(\ell(r)=\alpha\), \(p(r)=\beta\), and
\[
A_{\alpha,\beta}(y,d)=r\frac{y+\alpha}{\alpha+\beta},
\]
so the slope is \((y+\alpha)/(\alpha+\beta)\).  The formulas are
continuous at the transition points, hence the stated Lipschitz bound
follows by splitting an interval at those points.

If \(y\le\beta/2\), then
\[
\frac{y+\alpha}{\alpha+\beta}
\le \frac{\beta/2+\alpha}{\alpha+\beta}
\le\frac34,
\]
because \(\alpha\le\beta\).
\end{proof}

\section{The countable construction}\label{sectkey}

We call the bound functions $a=-\alpha$ and $b=\beta$, so that we manipulate nonnegative quantities $\alpha,\beta$. By switching the sign of coordinates where $\alpha(n)>\beta(n)$, we can suppose that we have the extra hypothesis that $\alpha(n)\leq \beta(n)$ for all $n$. Thus our objective is to find a homeomorphism $F:B^+(\mathbb{N})\longrightarrow B(\mathbb{N},-\alpha,\beta)$ that preserves the norm, assuming that $\alpha,\beta:\mathbb{N}\longrightarrow [0,1]$ are such that $\sum_n \beta_n=\infty$.  Removing trivial coordinates, we can also assume that $\beta(n)>0$ for all $n$. 
For each \(n\), write
\[
\Lambda_n=\Lambda_{\alpha_n,\beta_n},
\qquad
\Omega_n=(A_n,B_n)=\Omega_{\alpha_n,\beta_n}.
\]

Let \(x=(x_1,x_2,\ldots)\in B^+(\mathbb{N})\).  Start with \(c_0=x_1\), and
recursively define
\[
(y_n,c_n)=\Lambda_n(c_{n-1},x_{n+1})
\qquad(n\ge1).
\]
Set
\[
F(x)=(y_1,y_2,\ldots).
\]
Each \(y_n\) depends only on \(x_1,\ldots,x_{n+1}\) through a continuous formula, so \(F\) is
continuous for the product topologies.  By Proposition~\ref{prop:planar-homeomorphism},
\[
|y_n|+c_n=c_{n-1}+x_{n+1}
\qquad(n\ge1).
\]
Consequently,
\[
\sum_{k=1}^n |y_k|+c_n=\sum_{k=1}^{n+1}x_k
\qquad(n\ge1). \tag{1}
\]
In particular \(\sum_{k=1}^n|y_k|\le1\) for every \(n\).

\medskip
\noindent\textbf{Claim 1.}
For every \(x\in B^+(\mathbb{N})\), the sequence $(c_n)$ defined above satisfies \(\lim c_n = 0\).

\smallskip
\noindent\emph{Proof of Claim 1.}
Suppose not.  Then there is \(\eps>0\) such that \(c_n>\eps\) for
infinitely many \(n\).  Choose \(N\) so large that
\[
\sum_{j>N}x_j<\frac{\eps}{4}.
\]
Since \(c_n\le c_{n-1}+x_{n+1}\), once \(c_q\le\eps/2\) for some
\(q\ge N\), every later \(c_m\) is less than \(\eps\). This would contradict our assumption,  hence in fact
\[
c_n>\frac{\eps}{2}
\qquad(n\ge N).
\]
For \(n>N\), writing \(r_n=c_{n-1}+x_{n+1}\) we have
\[
\frac{c_{n-1}}{r_n}\ge \frac{c_{n-1}}{c_{n-1}+\varepsilon/4} \ge \frac{c_{n-1}}{c_{n-1}+c_{n-1}/2} \ge \frac23.
\]
Writing \(\ell_n=\min\{\alpha_n,r_n\}\) and
\(p_n=\min\{\beta_n,r_n\}\), the formula for \(\Lambda_n\) in Section~\ref{sectionplane} gives
\[
y_n=-\ell_n+(\ell_n+p_n)\frac{c_{n-1}}{r_n}
\ge -\ell_n+\frac23(\ell_n+p_n)
=\frac{2p_n-\ell_n}{3}
\ge\frac{p_n}{3},
\] where the last inequality comes from our extra hypothesis at the beginning of this section that $\alpha\leq \beta$.
Since \(r_n>\eps/2\),
\[
p_n\ge\min\left\{\beta_n,\frac{\eps}{2}\right\}.
\]
Thus
\[
|y_n|\ge\frac13\min\left\{\beta_n,\frac{\eps}{2}\right\}
\qquad(n>N).
\]
But \(\sum_n\beta_n=\infty\), and therefore
\(\sum_n\min\{\beta_n,\eps/2\}=\infty\), contradicting
\(\sum_n|y_n|\le1\).  This proves the claim. \claimqed

Letting \(n\to\infty\) in (1), Claim~1 gives
\[
\sum_{n=1}^\infty |F(x)_n|=
\sum_{n=1}^\infty x_n.
\]
So \(F(x)\in B(\mathbb{N},-\alpha,\beta)\), and \(F\) preserves the \(\ell_1\)-norm.

It remains to prove that \(F\) is bijective. Since $F$ is a continuous mapping between compact Hausdorff spaces, this will show that $F$ is a homeomorphism.  We fix
\(y=(y_1,y_2,\ldots)\in B(\mathbb{N},-\alpha,\beta)\). In the first place, we show that the auxiliary sequence $(c_n)$ is uniquely determined by $y$.

\medskip
\noindent\textbf{Claim 2.}
For every \(y\in B(\mathbb{N},-\alpha,\beta)\) there is a unique sequence
\((c_n)_{n\ge0}\) in \([0,1]\) satisfying \[
c_{n-1}=A_n(y_n,c_n)
\qquad(n\ge1). \tag{2}
\] 

\smallskip
\noindent\emph{Proof of Claim 2.}
For \(N\ge1\), define a finite backward sequence by
\[
c_N^{(N)}=0,
\qquad
c_{k-1}^{(N)}=A_k(y_k,c_k^{(N)})
\quad(k=N,N-1,\ldots,1).
\]
This is legitimate.  Indeed, the preservation of the norm by \(\Omega_k\) gives,
by backward induction,
\[
c_m^{(N)}\le \sum_{k=m+1}^{N}|y_k|
\qquad(0\le m<N).
\]
Hence \(|y_m|+c_m^{(N)}\le \sum_{k=m}^{N}|y_k|\le1\), so the pair
\((y_m,c_m^{(N)})\) lies in the domain of \(\Omega_m\).

Put
\[
\lambda_k=
\lambda_{\alpha_k,\beta_k}(y_k)=
\max\left\{\frac12,\frac{y_k+\alpha_k}{\alpha_k+\beta_k}\right\}.
\]
Remember from Proposition~\ref{prop:lipschitz-planar} that \(0\le\lambda_k\le1\), and we have
\(\lambda_k\le3/4\) whenever \(y_k\le\beta_k/2\).  The exceptional set
\[
E=\{k\in\N:y_k>\beta_k/2\}
\]
has finite \(\beta\)-sum, because
\[
\sum_{k\in E}\beta_k\le2\sum_{k\in E}y_k\le2\sum_{k=1}^\infty|y_k|\le2.
\]
Since \(\sum_k\beta_k=\infty\), the complement \(\N\setminus E\) is
infinite.  Therefore, for each fixed \(m\),
\[
\prod_{k=m+1}^{N}\lambda_k\longrightarrow0
\qquad(N\to\infty).
\]
If \(N'>N>m\), repeated use of Proposition~\ref{prop:lipschitz-planar}
gives
\[
|c_m^{(N')}-c_m^{(N)}|
\le
\left(\prod_{k=m+1}^{N}\lambda_k\right)
|c_N^{(N')}-c_N^{(N)}|
\le
\prod_{k=m+1}^{N}\lambda_k.
\]
Thus \((c_m^{(N)})_{N>m}\) is Cauchy.  Define
\[
c_m=\lim_{N\to\infty}c_m^{(N)}.
\]
Passing to the limit in the finite recursion gives (2).

The same estimate proves uniqueness.  If \((c_n)\) and \((c_n')\) both
solve (2), then
\[
|c_m-c_m'|\le
\left(\prod_{k=m+1}^{N}\lambda_k\right)|c_N-c_N'|
\le
\prod_{k=m+1}^{N}\lambda_k,
\]
and the right-hand side tends to \(0\).

 \claimqed

Given \(y\in B(\mathbb{N},-\alpha,\beta)\), let \((c_n)\) be the sequence from
Claim~2 and define
\[
x_1=c_0,
\qquad
x_{n+1}=B_n(y_n,c_n)
\quad(n\ge1).
\]
Then
\[
(c_{n-1},x_{n+1})=\Omega_n(y_n,c_n),
\]
so applying the forward recursion to this \(x\) gives back \(y\).  Also
\[
c_m\le\sum_{k=m+1}^\infty |y_k|,
\]
which follows by passing to the limit in the corresponding estimate for
\(c_m^{(N)}\).  Hence \(c_m\to0\).  Telescoping the norm-preservation identities $c_{n-1}+x_{n+1} = |y_n|+c_n$ gives
\[
\sum_{k=1}^{N+1}x_k=
\sum_{k=1}^{N}|y_k|+c_N,
\]
and therefore \(\sum_kx_k=\sum_k|y_k|\le1\).  Thus \(x\in B^+(\mathbb{N})\). The uniqueness in Claim~2 shows that no other \(x\) is mapped to \(y\). 
This finishes the proof of Theorem~\ref{thm:separable}.

\section{The uncountable construction}\label{sectglue}

We shall use the following elementary partition lemma.

\begin{lemma}\label{lem:partition}
Let \(\kappa\) be an uncountable cardinal, and let \((\beta_i)_{i\in\kappa}\)
be a family of positive real numbers.  Then \(\kappa\) can be partitioned
into countably infinite sets
\[
\kappa=\bigcup_{\xi<\kappa}S_\xi
\]
such that
\[
\sum_{i\in S_\xi}\beta_i=\infty
\qquad(\xi<\kappa).
\]
\end{lemma}

\begin{proof}
Construct pairwise disjoint countably infinite sets \(T_\xi\), \(\xi<\kappa\),
by transfinite recursion.  At stage \(\xi\), fewer than \(\kappa\) indices
have been used, so the remaining set has cardinality \(\kappa\).  Since it
is the union of the sets \(\{i:\beta_i\ge1/m\}\), \(m=1,2,\ldots\), one of
these sets is infinite inside the remaining set.  Choose a countably
infinite subset and call it \(T_\xi\).  Its \(\beta\)-sum is divergent. Once this recursion is finished, the leftover set \(L = \kappa\setminus \bigcup_{\xi<\kappa} T_\xi\) has cardinal at most \(\kappa\), so write it as a
disjoint union \(L=\bigcup_{\xi<\kappa}L_\xi\), where each \(L_\xi\) is
countable, possibly empty.  Then
\[
S_\xi=T_\xi\cup L_\xi
\]
is the required partition.
\end{proof}

We now prove Theorem~\ref{thm:classification} for uncountable $\kappa$. For $\kappa$ countable, it is a particular case of the aforementioned Keller-Klee Theorem, and for $\kappa$ finite it is elementary topology. By \cite[p.~339, Proposition~1]{AvilesKalenda2010},
we may suppose that
\[
-1\le a(i)\le0<b(i)\le1,
\qquad
|a(i)|\le b(i)
\quad(i\in\kappa).
\]
Put
\[
\alpha_i=-a(i),
\qquad
\beta_i=b(i).
\]
Then \(0\le\alpha_i\le\beta_i\le1\) and \(\beta_i>0\).

By Lemma~\ref{lem:partition}, choose a partition
\[
\kappa=\bigcup_{\xi<\kappa}S_\xi
\]
into countably infinite blocks such that
\[
\sum_{i\in S_\xi}\beta_i=\infty
\qquad(\xi<\kappa).
\]
Theorem~\ref{thm:separable}, applied to each $S_\xi$ instead of $\mathbb{N}$, gives a
homeomorphism
\[
F_\xi:
\left\{u\in[0,1]^{S_\xi}:\sum_{i\in S_\xi}u_i\le1\right\}
\longrightarrow
\left\{v\in\prod_{i\in S_\xi}[-\alpha_i,\beta_i]:
\sum_{i\in S_\xi}|v_i|\le1\right\}
\]
such that
\[
\sum_{i\in S_\xi}|F_\xi(u)_i|=
\sum_{i\in S_\xi}u_i.
\]

Define
\[
\Phi:B^+(\kappa)\longrightarrow B(\kappa,a,b)
\]
block by block:
\[
\Phi(x)|_{S_\xi}=F_\xi(x|_{S_\xi}).
\]
The norm identity on each block gives
\[
\sum_{i\in\kappa}|\Phi(x)_i|=
\sum_{i\in\kappa}x_i\le1,
\]
so \(\Phi(x)\in B(\kappa,a,b)\).  The inverse is obtained by applying
\(F_\xi^{-1}\) on each block.  Both \(\Phi\) and its inverse are continuous
because continuity into a product is coordinatewise.  Therefore
\(
B(\kappa,a,b)\) is homeomorphic to  \(B^+(\kappa)\). So all spaces $B(\kappa,a,b)$ as in the statement of Theorem~\ref{thm:classification} are indeed homeomorphic and the proof is finished.

\end{document}